\documentclass{amsart}
\usepackage{amssymb}
\usepackage{amsmath}
\usepackage{latexsym}
\usepackage{eepic}
\usepackage{epsfig}
\usepackage{graphicx}
\usepackage{pb-diagram,lamsarrow,pb-lams,amscd}
\usepackage{eufrak}
\usepackage{mathrsfs}
\usepackage[dutch,english]{babel}
\usepackage[all,cmtip]{xy}

\DeclareMathAlphabet{\mathpzc}{OT1}{pzc}{m}{it}
\newtheorem{theorem}{Theorem}[section]
\newtheorem{theorem-definition}[theorem]{Theorem-Definition}
\newtheorem{lemma-definition}[theorem]{Lemma-Definition}
\newtheorem{definition-prop}[theorem]{Proposition-Definition}
\newtheorem{corollary}[theorem]{Corollary}
\newtheorem{prop}[theorem]{Proposition}
\newtheorem{lemma}[theorem]{Lemma}
\newtheorem{cor}[theorem]{Corollary}
\newtheorem{definition}[theorem]{Definition}

\newenvironment{remark}{\vspace{4pt}\noindent\textbf{Remark.}}{\qed\vspace{4pt}}

\newcommand{\N}{\ensuremath{\mathbb{N}}}
\newcommand{\Z}{\ensuremath{\mathbb{Z}}}
\newcommand{\Q}{\ensuremath{\mathbb{Q}}}

\newcommand{\C}{\ensuremath{\mathbb{C}}}

\newcommand{\A}{\ensuremath{\mathbb{A}}}

\renewcommand{\C}{\ensuremath{\mathbb{C}}}

\renewcommand{\A}{\ensuremath{\mathbb{A}}}

\newcommand{\Spec}{\ensuremath{\mathrm{Spec}\,}}
\newcommand{\Spf}{\ensuremath{\mathrm{Spf}\,}}

\numberwithin{equation}{section}

\begin{document}
\title{Greenberg approximation and the geometry of arc spaces}
\author[Johannes Nicaise]{Johannes Nicaise}
\address{Johannes Nicaise\\  Universit\'e Lille 1\\
Laboratoire Painlev\'e, CNRS - UMR 8524\\ Cit\'e Scientifique\\59655 Villeneuve d'Ascq C\'edex \\
France} \email{johannes.nicaise@math.univ-lille1.fr}
\author{Julien Sebag}
\address{Julien Sebag\\ Universit\'e Bordeaux 1,
IMB, Laboratoire A2X \\ 351 cours de la lib\'eration\\ 33405
Talence cedex, France} \email{julien.sebag@math.u-bordeaux1.fr}
\begin{abstract}
We study the differential properties of generalized arc schemes
and geometric versions of Kolchin's Irreducibility Theorem over
arbitrary base fields. As an intermediate step, we prove an
approximation result for arcs by algebraic curves.
\end{abstract}
\maketitle
\section{Introduction}

In this article, we study geometric, topological and differential
properties of \textit{generalized arc schemes}.

\bigskip

One of the questions we deal with is their irreducibility. Let $k$
be a field. In differential algebra, \textit{Kolchin's
Irreducibility Theorem} states that, for any prime ideal $I$ in an
algebra of finite type over a field $k$ of characteristic zero,
the radical differential ideal $\{I\}$ associated to $I$ is again
prime \cite[Prop.~10,~p. 200]{kolchin}.

In \cite[3.3]{NiSe1}, we observed that the arc space of a
$k$-scheme of finite type $X$ can be constructed in terms of
differential algebra. Roughly speaking, the arc space associated
to $X$ is a $k$-scheme $\mathcal{L}(X)$ which parametrizes
$k[[t]]$-points on $X$. It contains deep information on the
structure of the singularities of $X$ and plays a fundamental role
in the theory of motivic integration. In this setting, Kolchin's Irreducibility Theorem states that
$\mathcal{L}(X)$ is irreducible if $X$ is. We'll refer to this
result as the \textit{Arc Scheme Irreducibility Theorem}. We gave
a purely geometric proof of this theorem, using resolution of
singularities, and a counterexample to show that it does not
extend to positive characteristic \cite[Rmq.\,1]{NiSe1}.

More recently, in \cite[2.9]{Reguera}, Reguera established a
modified form of the Arc Scheme Irreducibility Theorem, when $k$
is a \textit{perfect} field of positive characteristic.
 In the present article, we give a counterexample to show that her result fails
 if $k$ is imperfect (Theorem \ref{count}), and we show how it can be adapted
 for arbitrary fields (Theorem \ref{gen case}). Our result states that, for any field $k$ and any
  $k$-scheme of
finite type $X$, there exists a natural bijective correspondence
between the set of geometrically reduced irreducible components of
$X$, and the set of irreducible components of
$\mathcal{L}(X)\setminus \mathcal{L}(nSm(X))$. Here $nSm(X)$
denotes the non-smooth locus of $X$ over $k$. The results
presented in this article
 clarify the status of the Arc Scheme Irreducibility Theorem over
 arbitrary fields and incorporate all previously known cases.

\bigskip

As an intermediate result of independent interest, we prove that
points on $\mathcal{L}(X)$ can be approximated by algebraic curves
$C$ on $X$, in a sense which is made precise in Theorem
\ref{approx}. Since the topology of $\mathcal{L}(C)$ is easier to
control (Lemma \ref{curve}) one can use Theorem \ref{approx} to
study the topology of $\mathcal{L}(X)$ (see for instance
Proposition \ref{conn}).
 Our approximation result is based on the following
fundamental theorem by Greenberg \cite[Thm.\,1]{Gr} (we state it
in a slightly different, but equivalent form).

\begin{theorem}[Greenberg's Approximation Theorem]\label{grapp}
Let $R$ be an excellent henselian discrete valuation ring with
uniformizer $\pi$. For any $R$-scheme of finite type $X$ there
exists an integer $a\geq 1$ such that for any integer $\nu\geq 1$
the images of the natural maps
$$X(R/(\pi^{a\nu}))\rightarrow X(R/(\pi^\nu))\mbox{ and }X(R)\rightarrow X(R/(\pi^\nu))$$
coincide.
\end{theorem}

\bigskip

Besides irreducibility, it is natural to ask which other
properties of schemes are preserved by the arc space functor
(reducedness, noetherianity, connectedness,\ldots). In Section
\ref{sec-arc} we obtain some results in this direction, placing
ourselves in the most general setting: arc schemes
$\mathcal{L}(X/S)$ of arbitrary morphisms of schemes $X\rightarrow
S$. We call these objects \textit{generalized
  arc schemes}. Working on this level of generality is
useful, for instance, if one wants to consider the \textit{scheme
of wedges} of a relative scheme $X\rightarrow S$ as the iterated
arc scheme $\mathcal{L}(\mathcal{L}(X/S)/S)$ of $X/S$,
\textit{i.e.} if one studies infinitesimal deformations of arcs
(Definition \ref{wedge}). The wedge scheme plays an important role
in the study of the Nash problem \cite[5.1]{reguera-curve}.
Moreover, extending the theory to arbitrary relative schemes
yields a very natural proof of the Arc Scheme Irreducibility
Theorem for arbitrary schemes over a field of characteristic zero
(Theorem \ref{Qkol}). Our proof does not use resolution of
singularities and relies on the interpretation of arc schemes in
terms of differential algebra (Corollary \ref{arc-diff}).

 We show that
the generalized arc schemes have many properties with a
differential flavor, even in positive characteristic.
 Our main theorem in this direction is the characterization of
formally unramified morphisms (Theorem \ref{diff}(1)). These
differential properties are then applied in our study of the
geometry of arc schemes.

\bigskip

To conclude this introduction, we give a survey of the structure
of the paper. Section \ref{sec-arc} studies the differential
properties of arc schemes. We recall the general construction of
arc schemes in Section \ref{sec-arc-def}, we establish some basic
properties in Section \ref{stable}, and we develop the relations
between the geometry of arc schemes and the differential
properties of morphisms of schemes in Section \ref{sec-arc-diff}.

In Section \ref{sec-top} we focus on the topological properties of
arc schemes, relying on the results we proved in Section
\ref{sec-arc}. Section \ref{sec-top-diff} contains some
preliminaries. In Section \ref{sec-top-arcdiff} we interprete the
arc scheme in terms of differential algebra, and we use this
interpretation in Section \ref{sec-top-Qkol} to give a short
geometric proof of the Arc Scheme Irreducibility Theorem for
arbitrary schemes over a field of characteristic zero. Section
\ref{sec-top-wedge} gives an application of this result to wedge
schemes. Next, we prove our approximation result for arcs by
algebraic curves in Section \ref{sec-curves}. This result is used
in the topological study of the arc scheme in Section
\ref{sec-decomp}, where we prove various forms of the Arc Scheme
Irreducibility Theorem over arbitrary base fields. At the end of
the section, we show that there exists, over any imperfect field
$k$, a regular irreducible $k$-variety $X$ such that
$\mathcal{L}(X)$ is not irreducible. This shows that the statement
of \cite[2.9]{Reguera} does not extend to imperfect fields.

\subsection*{Notation}

For any field $k$, a $k$-variety is a reduced separated $k$-scheme
of finite type. A $k$-curve is a $k$-scheme of finite type of pure
dimension one. We do not demand it to be separated, nor reduced.

 We denote by $(\cdot)_{red}$ the endofunctor on the category
of schemes mapping a scheme $S$ to its maximal reduced closed
subscheme $S_{red}$. For any scheme $S$, an $S$-algebra (resp.
$S$-field) is a ring (resp. field) $A$ together with a morphism of
schemes $\Spec A\rightarrow S$.

For any field $k$ and any $k$-scheme $S$, we denote by $S_{alg}$
the set of points on $S$ whose residue field is algebraic over
$k$. Any morphism of $k$-schemes $T\rightarrow S$ maps $T_{alg}$
to $S_{alg}$, so $(\cdot)_{alg}$ defines a functor from the
category of $k$-schemes to the category of sets. If $k^{alg}$ is
an algebraic closure of $k$, then $S_{alg}$ is the image of the
natural map $S(k^{alg})\rightarrow S$. If $S$ is of finite type
over $k$, then $S_{alg}$ coincides with the set of closed points
of $S$.

If $X$ is a $k$-scheme of finite type over $k$, we denote by
$Reg(X)$ the set of regular points of $X$, and by $Sm(X)$ the set
of points where the structural morphism $X\rightarrow \Spec k$ is
smooth. These are open subsets of $X$ and we endow them with the
induced scheme structure. The complements of $Reg(X)$ and $Sm(X)$
in $X$, with their reduced induced closed subscheme structure, are
denoted by $Sing(X)$ (the singular locus of $X$), resp. $nSm(X)$
(the non-smooth locus of $X$). We say that $X$ has isolated
singularities if $Sing(X)$ is a finite set of points.

For any scheme $S$ and  any integer $n\geq 0$, we put
$S_n=S\times_{\Z} \Z[t]/(t^{n+1})$.
\section{Arc spaces}\label{sec-arc}

\subsection{Definition}\label{sec-arc-def}
We recall the definition of the arc scheme functor, for arbitrary
relative schemes.

\vspace{5pt} \noindent $\ast$  Let $S$ be any scheme, and $X$ any
$S$-scheme. By \cite[7.6.4]{neron}, the Weil restriction
$$\prod_{S_n/S}(X\times_S S_n)$$ is representable by a $S$-scheme, which we denote by
$\mathcal{L}_n(X/S)$. To be precise, the conditions in the
statement of \cite[7.6.4]{neron} are not necessarily fulfilled by
the $S_n$-scheme $X\times_S S_n$; however, going through the
proof, one sees that one only has to verify that for any geometric
point $z$ of $S$, the image of any morphism of schemes
$z\times_{S}S_n\rightarrow X\times_{S} S_n$ is contained in an
affine open subscheme of $X\times_{S} S_n$. This is trivial, since
 $z\times_{S}S_n$ is a point. Observe that $\mathcal{L}_0(X/S)$ is
 canonically isomorphic to $X$, and that $\mathcal{L}_n(S/S)$ is canonically isomorphic to $S$ for all $n\geq
 0$. If $S=\Spec A$, we also write
 $\mathcal{L}_n(X/A)$ instead of $\mathcal{L}_n(X/S)$.


\vspace{5pt} \noindent $\ast$  By functoriality of the Weil
restriction, $\mathcal{L}_n(\cdot/S)$ defines an endofunctor on
the category of $S$-schemes. By the proof of \cite[7.6.4]{neron},
$\mathcal{L}_n(X/S)$ is affine if $X$ and $S$ are affine. By
\cite[7.6.2]{neron}, the functor $\mathcal{L}_n(\cdot/S)$ respects
open, resp. closed immersions. By \cite[7.6.5]{neron},
$\mathcal{L}_n(X/S)$ is separated, resp. of finite presentation,
resp. smooth over $S$, if the same holds for $X$.

\vspace{5pt} \noindent $\ast$  For $m\geq n$ the closed immersion
$S_n\rightarrow S_m$ defined by reduction modulo $t^{n+1}$ induces
a natural morphism of schemes
$\pi^m_n:\mathcal{L}_m(X/S)\rightarrow \mathcal{L}_n(X/S)$. The
morphisms $\pi^m_n$ are affine, so that the projective limit
$$\mathcal{L}(X/S)=\lim_{\stackrel{\longleftarrow}{n}}\mathcal{L}_n(X/S)$$
exists in the category of schemes. We denote the natural
projection morphisms by $\pi_n:\mathcal{L}(X/S)\rightarrow
\mathcal{L}_n(X/S)$.

\begin{definition}
The scheme $\mathcal{L}_n(X/S)$ is called the $n$-jet scheme of
$X/S$, and $\mathcal{L}(X/S)$ is called the arc scheme of $X/S$.
The morphisms $\pi_n$ and $\pi^m_n$ are called truncation
morphisms.
\end{definition}

By the canonical isomorphism $\mathcal{L}_0(X/S)\cong X$, the
truncation morphisms $\pi^n_0$ and $\pi_0$ endow
$\mathcal{L}_n(X/S)$ and $\mathcal{L}(X/S)$ with a natural
structure of $X$-scheme. To uniformize notation, we will often put
$\mathcal{L}_{\infty}(\cdot)=\mathcal{L}(\cdot)$ and
$\pi^{\infty}_m=\pi_m$. We extend the usual ordering on $\N$ to
$\N\cup\{\infty\}$ by imposing that $\infty\geq n$ for all $n$ in
$\N\cup\{\infty\}$.

\vspace{5pt} \noindent $\ast$  It follows immediately from the
definition that for any morphism of schemes $X\rightarrow S$, the
arc scheme $\mathcal{L}(X/S)$ represents the functor from the
category of $S$-algebras to the category of sets sending a
$S$-algebra $A$ to the set $Hom_S(\Spf A[[t]], X)$. If $X$ is
affine, then the completion map
\begin{equation}\label{form-alg}
Hom_S(\Spec A[[t]], X)\rightarrow Hom_S(\Spf A[[t]],
X)\end{equation} is bijective and $\mathcal{L}(X/S)$ represents
the functor $A\mapsto X(A[[t]])$. It is an interesting open
question whether this property extends to all schemes $X$ of
finite type over $S=\Spec k$ with $k$ a field; this is the case
iff the functor $A\mapsto X(A[[t]])$ is a sheaf for the Zariski
topology on the category of $k$-algebras.

If $X\rightarrow S$ is any morphism of schemes and $A$ is a
\textit{local} $S$-algebra, then (\ref{form-alg}) is still a
bijection since any morphism $\Spec A[[t]]\rightarrow X$ factors
through an affine open subscheme of $X$. Hence, there exists a
natural bijection $\mathcal{L}(X/S)(A)=X(A[[t]])$. In particular,
for any $S$-field $F$, we have $\mathcal{L}(X/S)(F)= X(F[[t]])$.

\vspace{5pt} \noindent $\ast$  If $h:Y\rightarrow X$ is a morphism
of $S$-schemes, then the natural morphisms of $S$-schemes
$\mathcal{L}_n(h/S):\mathcal{L}_n(Y/S)\rightarrow
\mathcal{L}_n(X/S)$ commute with the truncation morphisms and
define a natural morphism of $S$-schemes
$\mathcal{L}(h/S):\mathcal{L}(Y/S)\rightarrow \mathcal{L}(X/S)$ by
passing to the limit, so $\mathcal{L}(\cdot/S)$ defines an
endofunctor on the category of $S$-schemes. If $h$ is a closed
immersion, then so is $\mathcal{L}(h/S)$, since it is a projective
limit of closed immersions $\mathcal{L}_n(h/S)$. It is easily seen
that $\mathcal{L}(\cdot/S)$ also respects open immersions; see
Theorem \ref{diff}(3) for a more general statement.

\vspace{5pt} \noindent $\ast$  The tautological morphism
$$X\rightarrow \prod_{S_n/S}(X\times_S S_n)$$ defines a section
$\tau^n_{X/S}:X\rightarrow \mathcal{L}_n(X/S)$ for the projection
morphism $\pi^n_0:\mathcal{L}_n(X/S)\rightarrow X$ for each $n\geq
0$, and by passing to the limit, we get a section
$\tau_{X/S}:X\rightarrow \mathcal{L}(X/S)$ for
$\pi_0:\mathcal{L}(X/S)\rightarrow X$ which sends a point $x$ of
$X$ to the constant arc at $x$.
The truncation morphisms $\pi^n_0$ and $\pi_0$ are affine, hence
separated, which implies that the sections $\tau^n_{X/S}$ and
$\tau_{X/S}$ are closed immersions.

\vspace{5pt} \noindent $\ast$  A morphism of schemes $T\rightarrow
S$ induces a natural base change morphism
$$\mathcal{L}_n(X/S)\times_S T\rightarrow\mathcal{L}_n(X\times_S
T/T)$$ for any $S$-scheme $X$, and a natural forgetful morphism
$$\mathcal{L}_n(Y/T)\rightarrow \mathcal{L}_n(Y/S)$$ 
for any $T$-scheme $Y$. These morphisms are compatible with the
truncation morphisms $\pi^m_n$. Taking limits, we get similar
morphisms on the level of arc spaces. The forgetful morphism
$$T\cong \mathcal{L}_n(T/T)\rightarrow \mathcal{L}_n(T/S)$$
coincides with $\tau^n_{T/S}$ for $n\in \N$ and with $\tau_{T/S}$
for $n=\infty$. In particular, it is a closed immersion.

\vspace{5pt} \noindent $\ast$  We'll consider two natural
topologies on $\mathcal{L}(X/S)$. The \textit{Zariski topology} on
the scheme $\mathcal{L}(X/S)$ coincides with the limit topology
w.r.t. the Zariski topology on the schemes $\mathcal{L}_n(X/S)$,
by \cite[8.2.9]{ega4.3}. Besides, we introduce the following
definition.

\begin{definition}[$t$-adic topology]
The \textit{$t$-adic topology} on $\mathcal{L}(X/S)$ is the limit
topology on $\mathcal{L}(X/S)$ w.r.t. the discrete topology on the
schemes $\mathcal{L}_n(X/S)$.
\end{definition}

\subsection{Basic properties}
\label{stable}

We establish some fundamental properties of the arc scheme.
\begin{prop}\label{basic}
Let $S$ be any scheme, let $X,Y,Z,T$ be schemes over $S$, and let
$W\rightarrow V$ be a morphism of $T$-schemes. Fix a value $n\in
\N\cup\{\infty\}$.
\begin{enumerate}
\item The functor $\mathcal{L}_n(\cdot/S)$ commutes with base
change: the natural base change morphism
$$\mathcal{L}_n(X/S)\times_S T\rightarrow\mathcal{L}(X\times_S T/T)$$ is an isomorphism.
 \item The functor $\mathcal{L}(\cdot/S)$
commutes with fibered products: for any $S$-morphisms
$X\rightarrow Z$ and $Y\rightarrow Z$, there is a natural
isomorphism
$$\mathcal{L}_n(X\times_Z Y/S)\cong
\mathcal{L}_n(X/S)\times_{\mathcal{L}_n(Z/S)}\mathcal{L}_n(Y/S)$$
\item There is a natural isomorphism
$$\mathcal{L}_n(W/T)\cong
\mathcal{L}_n(W/S)\times_{\mathcal{L}_n(V/S)}\mathcal{L}_n(V/T)$$
 \item The natural forgetful morphism
$$\mathcal{L}_n(W/T)\rightarrow \mathcal{L}_n(W/S)$$ is a closed immersion.
\item The natural morphism $X_{red}\rightarrow X$ induces an
isomorphism $\mathcal{L}(X_{red}/S)_{red}\rightarrow
\mathcal{L}(X/S)_{red}$. In particular, the natural morphism
$\mathcal{L}(X_{red}/S)\rightarrow \mathcal{L}(X/S)$ is a
homeomorphism. \item If $\mathcal{L}(X/S)$ is reduced, then $X$ is
reduced.


 \item
If the truncation morphism $\pi^n_0:\mathcal{L}_n(T/S)\rightarrow
T$ is an isomorphism then the forgetful morphism
$\mathcal{L}_n(W/T)\rightarrow \mathcal{L}_n(W/S)$ is an
isomorphism.
 Likewise, if
$(\pi^n_0)_{red}:\mathcal{L}_n(T/S)_{red}\rightarrow T_{red}$ is
an isomorphism then $\mathcal{L}_n(W/T)_{red}\rightarrow
\mathcal{L}_n(W/S)_{red}$ is an isomorphism.
\end{enumerate}
\end{prop}
\begin{proof}
Points (1), (2) and (3) are straightforward, and (4) follows from
(3) by taking $V=T$. So let us start with (5). Since
$$\mathcal{L}(X_{red}/S)(F)=X_{red}(F[[t]])=X(F[[t]])=\mathcal{L}(X/S)(F)$$
for any $S$-field $F$, we see that the natural closed immersion
$\mathcal{L}(X_{red}/S)\rightarrow \mathcal{L}(X/S)$ is bijective.
Hence, $\mathcal{L}(X_{red}/S)_{red}\rightarrow
\mathcal{L}(X/S)_{red}$ is an isomorphism.

(6) If $\mathcal{L}(X/S)$ is reduced, then the composition of the
natural section $\tau_{X/S}:X\rightarrow \mathcal{L}(X/S)$ with
the morphism $(\pi_0)_{red}:\mathcal{L}(X/S)\rightarrow X_{red}$
defines a left inverse $X\rightarrow X_{red}$ for the natural
closed immersion $X_{red}\rightarrow X$. This is only possible if
$X$ is reduced.



(7) follows from (3) by putting $V=T$.
\end{proof}
\begin{remark}
1. Proposition \ref{basic}(7) is reminiscent of the first
fundamental exact sequence for modules of differentials. In fact,
the first part of Proposition \ref{basic}(7) can be deduced from
the first fundamental exact sequence for Hasse-Schmidt derivations
\cite[2.1]{vojta}.

2. The converse of (6) does not hold. For instance, consider the
case where $k$ is a field of characteristic $2$, and put $S=\Spec
k(u)$ and  $X=\Spec k(u)[x]/(x^2+u)$. For a counterexample in
characteristic zero, consider the complex cusp $\Spec
\C[x,y]/(y^2-x^3)$ (see \cite[3.16]{Reguera}). It would be
interesting to find a characterization of the (complex) varieties
with reduced arc scheme, and more generally to understand the
geometric meaning of the non-reduced structure of the jet schemes
and the arc scheme.
%
\end{remark}

\subsection{Differential properties of the arc scheme}\label{sec-arc-diff}
 We'll show that the
structure of the arc scheme is closely related to the differential
properties of morphisms of schemes. First, we need an elementary
lemma.

\begin{lemma}\label{algclo}
Let $k$ be any field, let $k'$ be an algebraic extension of $k$,
and let $K$ be any field containing $k$. If $\varphi:k'\rightarrow
K[[t]]$ is a morphism of $k$-algebras, then the image of $\varphi$
is contained in $K$.
\end{lemma}
\begin{proof}
Suppose that $L$ is an algebraic field extension of $K$ inside
$K[[t]]$. It suffices to show that $K=L$. Let $\alpha=\sum_{i\geq
0}\alpha_i t^i$ be a non-zero element of $L$, with $\alpha_i\in K$
for all $i$, and denote by $p(x)\in K[x]$ its minimal polynomial
over $K$. Let $j\geq 0$ be the smallest index such that
$\alpha_j\neq 0$ and suppose that $j>0$. We will deduce a
contradiction.

 We may assume that $\alpha$ is
either separable or purely inseparable over $K$. In the first
case, $$0=p(\alpha)\equiv p(\alpha_0)+(\partial_xp)(\alpha_0)\cdot
\alpha_j\cdot t^j\mod t^{j+1}$$ which is impossible because
$\alpha_j\neq 0$ by assumption and
 $p$ is separable. In the second case, $p(x)$ is of the form
 $x^{p^m}-a$, with $p>0$ the characteristic of $k$, and
 $$0=p(\alpha)=\sum_{i\geq 0}(\alpha_i)^{p^m}t^{ip^m}-a$$ which
 yields the contradiction $\alpha_j=0$.
\end{proof}
\begin{theorem}\label{diff}
Let $U$ be any scheme, and let $h:T\rightarrow S$ be a morphism of
$U$-schemes.
\begin{enumerate}


\item The following are equivalent:

 (a) the morphism
$T\rightarrow S$ is formally unramified;

(b) there exists an integer $m\in \N^{\ast}\cup\{\infty\}$ such
that the truncation morphism
$\pi^m_0:\mathcal{L}_m(T/S)\rightarrow T$ is an isomorphism;

(c) $\pi^m_0:\mathcal{L}_m(T/S)\rightarrow T$ is an isomorphism
for each $m\in \N\cup\{\infty\}$.

\noindent If $T$ is Noetherian, then these properties are also
equivalent to:

(d) the arc scheme $\mathcal{L}(T/S)$ is Noetherian.

\item If $T\rightarrow S$ is formally smooth, then the following
properties hold:

(a) the morphism $\pi^m_n:\mathcal{L}_m(T/S)\rightarrow
\mathcal{L}_n(T/S)$ is
 surjective for each pair $m\geq n$ in $\N\cup\{\infty\}$,

(b) the morphism $\mathcal{L}_n(T/U)\rightarrow
\mathcal{L}_n(S/U)$ is formally smooth for each $n\in
\N\cup\{\infty\}$,
%

(c) the natural morphism
$$\mathcal{L}_m(T/U)\rightarrow \mathcal{L}_{m}(S/U)\times_{\mathcal{L}_n(S/U)}\mathcal{L}_n(T/U)$$
is surjective for each pair  $m\geq n$ in $\N\cup\{\infty\}$.

\item If $T\rightarrow S$ is formally \'etale, then the natural
diagram
$$\begin{CD}
\mathcal{L}_m(T/U)@>>> \mathcal{L}_m(S/U)
\\@V\pi^{m}_{n}VV @VV\pi^{m}_nV
\\ \mathcal{L}_n(T/U)@>>> \mathcal{L}_n(S/U)
\end{CD}$$ is Cartesian for each pair $m\geq n$ in $\N\cup\{\infty\}$.

\item Let $k\subset k'$ be $S$-fields.

(a) If $k'$ is separable over $k$, then $\mathcal{L}(\Spec
k'/S)\rightarrow \mathcal{L}(\Spec k/S)$ is surjective.

(b) If $k'$ is algebraic over $k$ then
$(\pi_0)_{red}:\mathcal{L}(\Spec k'/k)_{red}\rightarrow \Spec k'$
is an isomorphism. The converse holds if $k'$ is finite over a
separably generated extension of $k$.

 \item If $T\rightarrow S$ is locally of finite type, then the following are equivalent:

(a) the morphism $T\rightarrow S$ is quasi-finite;

(b) the morphism $(\pi_0)_{red}:\mathcal{L}(T/S)_{red}\rightarrow
T_{red}$ is an isomorphism;

(c) the morphism $(\pi_0)_{red}$ is of finite type.

\noindent If, moreover, $T$ is Noetherian, then these properties
are also equivalent to:

(d) the scheme $\mathcal{L}(T/S)_{red}$ is Noetherian.
\end{enumerate}
\end{theorem}
\begin{proof}
(1) The implications $(c)\Rightarrow (b)$ and $(c)\Rightarrow (d)$
 are trivial.

 It follows immediately from the definition that
$\mathcal{L}_1(T/S)$ is naturally isomorphic, as a $T$-scheme, to
the relative tangent scheme $\mathbf{Spec}\,(Sym(\Omega_{T/S}))$,
so $(c)\Rightarrow (a)\Rightarrow (b)$ by \cite[17.2.1]{ega4.4}.

Assume that $(a)$ holds. It is enough to prove $(c)$ for $m\in \N$
(the case $m=\infty$ follows by passing to the limit). The
property $(c)$ is equivalent to the property that for each
$S$-scheme $Z$, the natural map $Hom_S(Z_m,T)\rightarrow
Hom_S(Z,T)$ is a bijection. This map is always surjective, by the
existence of the natural section $Z_m\rightarrow Z$ for the
truncation morphism $Z\rightarrow Z_m$. It is injective by the
infinitesimal lifting criterion for formally unramified morphisms
\cite[17.1.1]{ega4.4}.

Now assume that $(b)$ holds. We will deduce $(a)$. We may assume
that $S=\Spec A$ and $T=\Spec B$ are affine. Suppose that
$T\rightarrow S$ is formally ramified. By \cite[17.2.1]{ega4.4},
this means that $\Omega^1_{B/A}\neq 0$, \textit{i.e.}  there
exists a $A$-algebra $C$ and a morphism of $A$-algebras
$\varphi:B\rightarrow C[u]/(u^2)$ whose image is not contained in
$C$. Put $C'=C[u]/(u^2)$. Composing $\varphi$ with the morphism of
$C$-algebras $$C'\rightarrow  C'[[t]]:u\mapsto u(1+t)$$ if
$m=\infty$ and with $$C'\rightarrow  C'[t]/(t^{m+1}):u\mapsto
u(1+t)$$ else, we get an element of $\mathcal{L}_m(T/S)(C')$ which
is not contained in the image of
$$\tau^m_{T/S}(C'):T(C')\rightarrow \mathcal{L}_m(T/S)(C')$$ This
contradicts the assumption that $\pi^m_0$, and hence the section
$\tau^m_{T/S}$, are isomorphisms.

Finally, suppose that $T$ is Noetherian and $(d)$ holds. We will
deduce $(b)$ with $m=\infty$. We may assume that $S=\Spec A$ and
$T=\Spec B$ are affine.
 Denote by $\mathcal{F}$ the contravariant functor from the category of
$A$-algebras to the category of $S$-schemes mapping an $A$-algebra
 $C$ to the scheme $\Spec C[[t]]$, and, for each $n>0$, denote by $j_n$ the natural
transformation $\mathcal{F}\rightarrow \mathcal{F}$ mapping $C$ to
the morphism
$$j_n(C):\Spec C[[t]]\rightarrow \Spec C[[t]]:t\mapsto t^n$$
 It is
easily seen that there exists for each $n>0$ a unique closed
immersion of $S$-schemes $\iota_n:\mathcal{L}(T/S)\rightarrow
\mathcal{L}(T/S)$ inducing the map
$$\iota_n(C):\mathcal{L}(T/S)(C)=T(C[[t]])\rightarrow T(C[[t^n]])\subset T(C[[t]]):x\mapsto x\circ j_n(C)$$
for each $A$-algebra $C$ (see \cite[3.8]{NiSe-weilres} for a more
general result). By the Ascending Chain Condition for ideals in a
Noetherian ring, any closed immersion of a Noetherian scheme into
itself is an isomorphism. Hence, $\iota_n$ is an isomorphism, and
the inclusion $T(C[[t^n]])\subset T(C[[t]])$ is a bijection for
each $n>0$ and any $A$-algebra $C$. On the other hand,
$\cap_{n>0}T(C[[t^n]])=T(C)$, so $T(C)\subset T(C[[t]])$ is a
bijection and $\tau_{T/S}:T\rightarrow \mathcal{L}(T/S)$ is an
isomorphism, inverse to $\pi_0$. \vspace{5pt}

(2,3) These are all easy formal consequences of the infinitesimal
lifting criteria for formally smooth, resp. formally \'etale
morphisms \cite[17.1.1]{ega4.4}.

\vspace{5pt}
 (4a) follows from (2c) (with $m=\infty$ and $n=0$) since $k'$ is
 formally smooth over $k$ by \cite[19.6.1]{ega4.1}.

 (4b) By the existence of the section $\tau_{T/S}$, which is a closed
 immersion, $(\pi_0)_{red}$ is an isomorphism iff it is a
 bijection on the level of underlying sets.

 Assume that $k'/k$ is algebraic, let $K$ be a
$k$-field, and $\varphi:k'\rightarrow K[[t]]$ a morphism of
$k$-algebras. It suffices to show that the image of $\varphi$ is
contained in $K$. This follows from Lemma \ref{algclo}.

Conversely, assume that $(\pi_0)_{red}$ is an isomorphism and that
$k'$ is finite over a separably generated extension $L$ of $k$. We
may assume that $L$ is separably closed in $k'$. Let
$K=k(u_i)_{i\in I}$ be a purely transcendental extension of $k$
inside $L$ such that $L/K$ is separable and algebraic. We have to
show that $I$ is empty; suppose the contrary. Consider the
morphism of $k$-algebras $\varphi:K\rightarrow K[[t]]$ mapping
$u_i$ to $u_i+t$ for each $i\in I$. By Hensel's Lemma, it extends
uniquely to a morphism $\varphi:L\rightarrow L[[t]]$ whose
composition with reduction modulo $t$ is the identity on $L$. If
$\alpha$ is a purely inseparable element of $k'$ over $L$, then
for $i\gg 0$,
 $\varphi$
extends uniquely to a morphism $L(\alpha)\rightarrow
L(\alpha)[[t^{p^{-i}}]]$ such that the composition with reduction
modulo $t^{p^{-i}}$ is the identity on $K(\alpha)$.
Reparametrizing by putting $t'=t^{p^{-i}}$ and continuing with $L$
replaced by $L(\alpha)$, we obtain an extension of $\varphi$ to a
morphism $k'\rightarrow k'[[t]]$ whose image is not contained in
$k'$. This contradicts the assumption that $(\pi_0)_{red}$ is an
isomorphism.

\vspace{5pt}

 (5) The implications $(b)\Rightarrow (c)\Rightarrow (d)$ are trivial. We noted already in the proof of (2)
 that $(\pi_0)_{red}$ is an isomorphism iff it is a
 bijection on the level of underlying sets. Hence, by Proposition \ref{basic}(1), all properties
 in the statement can be checked on the fibers and we may assume
 that $S=\Spec k$ with $k$ a field. Since all the properties are
 local on $T$, we may assume that $T$ is connected and $T=\Spec B$ with $B$ a
 $k$-algebra of finite type. By Proposition \ref{basic}(5) we may assume that
 $T$ is reduced.

If $(a)$ holds, then $B$ is a finite field extension of $k$, and
$(b)$ follows from Lemma \ref{algclo}.


 Now, assume that $(b)$ holds. In order to deduce $(a)$, we have to show that $B$ is a field.
If $P$ is a minimal prime ideal $B$, then $B$ is a field iff $B/P$
is a field, since $B$ is reduced and $\Spec B$ is connected.
Hence, we may assume that $B$ is a domain. Denote by $K$ its
quotient field. The fact that $\mathcal{L}(T/S)\rightarrow T$ is a
bijection implies that the same holds if we replace $T$ by $\Spec
K$, by (3). But $K/k$ is finitely generated, so (4) implies that
$K$ is algebraic over $k$, and we can conclude that $B=K$.

Finally, if $(d)$ holds, the arguments in the proof of
(1)$(d)\Rightarrow (b)$ (restricted to reduced algebras $C$) show
that $(b)$ holds.
\end{proof}
\begin{remark}
Property (4b) does not extend to the jet spaces
$\mathcal{L}_m(\Spec k'/k)$. For instance, let $k$ be an imperfect
field of characteristic $p$, pick an element $a$ in $k-k^p$, and
put $k'=k[x]/(x^p-a)$. Then $\mathcal{L}_n(\Spec k'
/k)_{red}\rightarrow \Spec k'$ is not an isomorphism for any $n\in
\N^\ast$, since $\mathcal{L}_n(\Spec k'/k)_{red}$ is the closed
subscheme of $\Spec k[x_0,\ldots,x_n]$ defined by the equations
$x_i=0$ for all $i\geq 0$ with $i\cdot p\leq n$.

For the converse implication in (4b), the condition that $k'$ is
finite over a separably generated extension of $k$ cannot be
omitted: there exist formally unramified field extensions $k'/k$
which are not algebraic. For instance, if $k$ is any field of
characteristic $p>0$ and $k'$ is the perfect closure of $k(u)$,
then $\Omega^1_{k'/\Z}=0$ because $d(a^p)=p\cdot a^{p-1}\cdot
da=0$ for any element $a$ of $k'$. Since $k'/k$ is not separable
if $k$ is not perfect, the same example shows that the converse of
(4a) is false.

Similar examples show that the condition that $T\rightarrow S$ is
locally of finite type cannot be dropped in (5), even if we
replace ``is quasi-finite'' by ``has discrete fibers'' in (5a).
For example, if $k$ is an algebraically closed field of
characteristic $p>0$ and $B$ is the $k$-algebra
$\cup_{i>0}k[{t}^{p^{-i}}]$, then $B$ is formally unramified over
$k$ but $\Spec B$ is not discrete.
\end{remark}
\begin{cor}\label{cor-ft}
If $k$ is any field and $X$ is a $k$-scheme of finite type, then
\begin{itemize}
\item the following are equivalent:

(a) the arc scheme $\mathcal{L}(X/k)$ is Noetherian;

(b) the truncation morphism $\pi_0:\mathcal{L}(X/k)\rightarrow X$
is an isomorphism;

(c) the scheme $X$ is \'etale over $k$.

\item the following are equivalent:

(a) the scheme $\mathcal{L}(X/k)_{red}$ is Noetherian;

(b) the morphism $(\pi_0)_{red}:\mathcal{L}(X/k)_{red}\rightarrow
X_{red}$ is an isomorphism;

(c) the scheme $X$ has dimension $0$.

\end{itemize}
\end{cor}
\begin{cor}\label{forget}
Let $k'/k$ a field extension, and fix a $k'$-scheme $X$.
\begin{itemize}
\item If $k'/k$ is algebraic,then the forgetful morphism
$$\mathcal{L}(X/k')_{red}\rightarrow \mathcal{L}(X/k)_{red}$$ is an
isomorphism.

\item If the field $k'$ is formally unramified over $k$ then the
forgetful morphism
$$\mathcal{L}_n(X/k')\rightarrow \mathcal{L}_n(X/k)$$ is an
isomorphism for each $n\in \N\cup\{\infty\}$.
\end{itemize}
\end{cor}
\begin{proof}
This follows immediately from Proposition \ref{basic}(7) and
Theorem \ref{diff}(1,4b).
\end{proof}
%

%
%



\section{Topology of arc spaces}
\label{sec-top}
\subsection{Basic topological properties of the arc
scheme}\label{sec-top-diff} We collect some elementary topological
properties of the arc scheme, which we'll need in the remainder of
the article.
\begin{prop}\label{prelim}
Let $X\rightarrow S$ be a morphism of schemes. We denote by
$\mathcal{I}(X)$ the set of irreducible components of $X$, and by
$\mathcal{C}(X)$ the set of connected components of $X$ (with
their reduced induced structure).
\begin{enumerate}
 \item
If $X$ is integral, and smooth over $S$, then $\mathcal{L}(X/S)$
is integral.
  \item For each element $X_i$
of $\mathcal{I}(X)$ (resp. $\mathcal{C}(X)$), the scheme
$\mathcal{L}(X_i/S)$ is a union of irreducible (resp. connected)
components of $\mathcal{L}(X/S)$. Moreover,
$$\mathcal{L}(X/S)=\cup_{X_i\in
\mathcal{I}(X)}\mathcal{L}(X_i/S)\mbox{ and
}\mathcal{L}(X/S)=\cup_{Y_j\in\mathcal{C}(X)}\mathcal{L}(Y_j/S)$$
(as topological spaces), and if $X_i$ and $X_j$ are distinct
elements of $\mathcal{I}(X)$ (resp. $\mathcal{C}(X)$) then
$\mathcal{L}(X_i/S)\not\subset \mathcal{L}(X_j/S)$. \item If
$S=\Spec k$, with $k$ a field, and $k'$ is an algebraic, purely
inseparable field extension of $k$, then the natural morphism
$h:\mathcal{L}(X'\times_k k'/k')\rightarrow \mathcal{L}(X/k)$
induced by the base change isomorphism in Proposition
\ref{basic}(1), is a homeomorphism.
\end{enumerate}
\end{prop}
\begin{proof}
(1)  If $X$ is smooth over $S$, it admits Zariski-locally an
\'etale morphism to affine space $\A^d_S$ by
\cite[17.11.4]{ega4.4}, and
$\pi^{n}_{0}:\mathcal{L}_n(\A^d_S/S)\rightarrow \A^d_S$ is
obviously a trivial $S$-fibration with fiber $\A^{(n+1)d}_S$ for
each $n\geq 0$. Hence, it follows from integrality of $X$ and
Theorem \ref{diff}(3) that $\mathcal{L}(X/S)$ is integral.

\if false Hence, we may assume $X$ is affine. In this case, we
have natural bijections
$$\mathcal{L}(X_{red})(A)=X_{red}(A[[t]])=X(A[[t]])=\mathcal{L}(X)(A)$$
for each reduced $k$-algebra $A$, which implies that the natural
morphism $\mathcal{L}(X_{red})_{red}\rightarrow
\mathcal{L}(X)_{red}$ is an isomorphism. \fi

(2) These properties follow easily from the fact that the arc
scheme functor $\mathcal{L}(\cdot)$ respects open (resp. closed)
immersions, and the fact that for any $k$-field $F$, any morphism
$\Spec F[[t]]\rightarrow X$ factors through an irreducible
component of $X$ since $F[[t]]$ is integral. For the last
statement, use the existence of the natural sections
$\tau_{X_i/S}$ and $\tau_{X_j/S}$.

(3) The natural projection $\mathcal{L}(X/k)\times_k k'\rightarrow
\mathcal{L}(X/k)$ is a homeomorphism by \cite[2.4.5]{ega4.2}.
\end{proof}


\subsection{Arc schemes and differential
algebra}\label{sec-top-arcdiff} In this section, we'll show that
arc schemes in characteristic zero admit a natural interpretation
in terms of differential algebra. Let $(A,\delta)$ be a
differential ring. By \cite[1.19]{gillet-diff} the forgetful
functor $For$ from the category of differential
$(A,\delta)$-algebras to the category of $A$-algebras has a left
adjoint $R\mapsto R^{\infty}$. By definition, there is a
tautological morphism of $A$-algebras $R\rightarrow
For(R^{\infty})$. In order to avoid confusion we adopt the
following notation: if $B$ and $C$ are differential
$(A,\delta)$-algebras, we'll write $Hom_A^{\delta}(B,C)$ for the
set of morphisms of differential $(A,\delta)$-algebras and
$Hom_A(B,C)$ for the set of morphisms of $A$-algebras
$Hom_A(For(B),For(C))$.

\begin{prop}\label{right}
Assume that $A$ is a $\Q$-algebra endowed with the trivial
derivation. Then the forgetful functor $For$ has a right adjoint
$R\mapsto R_{\infty}$. It maps an $A$-algebra $R$ to the
$A$-algebra $R_{\infty}=R[[t]]$ endowed with the usual derivation
$\delta=\partial_t$ with respect to the parameter $t$.
\end{prop}
\begin{proof}
A morphism of $A$-algebras $R\rightarrow S$ defines a morphism of
differential $A$-algebras $(R[[t]],\partial_t)\rightarrow
(S[[t]],\partial_t)$ in the obvious way, so the functor
$(\cdot)_{\infty}$ is well-defined. Let us check that it is right
adjoint to $For$. Let $(B,\delta)$ be any differential
$A$-algebra, and $\varphi:B\rightarrow R$ a morphism of
$A$-algebras. Then $\varphi$ defines a morphism of differential
$A$-algebras
$$B\rightarrow R_{\infty}:b\mapsto \sum_{i\geq 0}\frac{\varphi(\delta^{(i)}b)}{i!}t^i$$
Conversely, a morphism of differential $A$-algebras
$(B,\delta)\rightarrow R_{\infty}$ defines a morphism of
$A$-algebras $B\rightarrow R$ by composition with reduction modulo
$t$, and one checks that these correspondences define mutually
inverse bijections
$$Hom_A(B,R)=Hom_A^{\delta}((B,\delta),R_{\infty})$$
\end{proof}

\begin{cor}\label{arc-diff}
Let $A$ be a $\Q$-algebra endowed with the trivial derivation. For
any $A$-algebra $R$, the $R$-scheme $\Spec R^{\infty}$ is
canonically isomorphic to the arc scheme $\mathcal{L}(\Spec R/A)$.
\end{cor}
\begin{proof}
By Proposition \ref{right} and the fact that $(\cdot)^{\infty}$ is
left adjoint to $For$, we have for any $R$-algebra $S$ natural
bijections
\begin{eqnarray*}
Hom_A(For(R^{\infty}),S)&=&Hom_A^{\delta}(R^{\infty},S_\infty)
\\ &=&Hom_A(R,S[[t]])
\end{eqnarray*}
so there is a natural isomorphism of $A$-schemes
$\mathcal{L}(\Spec R/A)\cong \Spec R^{\infty}$, and one easily
checks that it is an isomorphism of $R$-schemes.
\end{proof}
\begin{cor}\label{forsmooth}
If $S$ is a $\Q$-scheme and $X$ is a formally smooth integral
$S$-scheme, then $\mathcal{L}(X/S)$ is integral.
\end{cor}
\begin{proof}
We may assume that $X$ and $S$ are affine, say $X=\Spec R$ and
$S=\Spec A$, and it suffices to show that $R^{\infty}$ is a
domain. This follows from the fact that $\Omega^1_{R/A}$ is a
projective $R$-module and the formal properties of the functor
$(\cdot)^{\infty}$; see \cite[1.27]{gillet-diff}.
\end{proof}

\subsection{Kolchin's Irreducibility Theorem for arbitrary
$\Q$-schemes}\label{sec-top-Qkol} In \cite[Thm.\,1]{gillet-diff}
it is shown that, if $(k,\delta)$ is a differential field of
characteristic zero and $R$ is an integral $k$-algebra, then
$R^{\infty}$ has a unique minimal prime ideal. This is a very
general form of Kolchin's Irreducibility Theorem. In view of
Corollary \ref{arc-diff}, it implies that the arc scheme
$\mathcal{L}(X/k)$ of an arbitrary irreducible scheme $X$ over a
field $k$ of characteristic zero is again irreducible. We will
give a short geometric proof of this result.

\begin{lemma}\label{valuation}
If $k$ is a field of characteristic zero, and $V$ a valuation ring
containing $k$, then $\mathcal{L}(\Spec V/k)$ is integral.
\end{lemma}
\begin{proof}
Consider the endofunctor $\mathcal{G}$ on the category of
$k$-algebras which sends a $k$-algebra $B$ to the $k$-algebra of
global sections of the $k$-scheme $\mathcal{L}(\Spec B/k)$. It
admits a right adjoint, sending a $k$-algebra $C$ to the
$k$-algebra $C[[t]]$~:  use either the fact that the functor
$\mathcal{G}$ is isomorphic to $For\circ(\cdot)^{\infty}$ by
Corollary \ref{arc-diff}, or the bijectivity of the completion map
(\ref{form-alg}) in the affine case.

Now we can follow the arguments in \cite[1.38]{gillet-diff}. As a
left adjoint, $\mathcal{G}$ commutes with direct limits. By
Zariski's Uniformization Theorem, $V$ can be written as a direct
limit of domains $B_i$ which are smooth over $k$. But
$\mathcal{G}(B_i)$ is a domain by Proposition \ref{prelim}(1), and
a direct limit of domains is again a domain.
\end{proof}

 If
$X\rightarrow S$ is a morphism of schemes, $K$ a $S$-field and
$\varphi:\Spec K[[t]]\rightarrow X$ an element of
$\mathcal{L}(X/S)(K)$, then we denote by $\varphi(0)$ the image of
the closed point of $\Spec K[[t]]$ in $X$, and by $\varphi(\eta)$
the image of the generic point of $\Spec K[[t]]$. Note that
$\varphi(0)=\pi_0(\varphi)$.

\begin{theorem}\label{Qkol} Let $X\rightarrow S$ be a morphism of schemes.
\begin{enumerate}
\item For any point $\varphi$ on $\mathcal{L}(X/S)$, there exists
a point $\psi$ in $\mathcal{L}(\varphi(\eta)/S)\subset
\mathcal{L}(X/S)$ such that $\varphi$ belongs to the Zariski
closure of $\psi$ in $\mathcal{L}(X/S)$.

\item If $k$ is a field of characteristic zero and $S=\Spec k$,
then the map
$$Y\mapsto \mathcal{L}(Y/k)_{red}$$ defines a bijective
correspondence between the irreducible components $Y$ of $X$ and
the irreducible components of $\mathcal{L}(X/k)$.

In particular, if $X$ is irreducible, then $\mathcal{L}(X/k)$ is
irreducible. Moreover, if $\xi$ is the generic point of $X$, then
$\mathcal{L}(\xi/k)$ is dense in $\mathcal{L}(X/k)$.
\end{enumerate}
\end{theorem}
\begin{proof}
To prove (1) one can copy the arguments in the first part of the
proof of \cite[2.12]{ishii-kollar}. We repeat them here for the
reader's convenience. Let $\varphi$ be any point of
$\mathcal{L}(X/S)$, with residue field $K$. It corresponds to a
morphism of $S$-schemes $\varphi:\Spec K[[t]]\rightarrow X$.
 Composing $\varphi$ with the morphism
$$\Spec K[[u,v]]\rightarrow \Spec K[[t]]:t\mapsto u+v$$ we get an
element $\psi$ of $X(K[[u,v]])$, \textit{i.e.}  an arc
$\zeta:\Spec K[[u]]\rightarrow \mathcal{L}(X)$. It satisfies
$\zeta(0)=\varphi$, and the point $\zeta(\eta)$ corresponds to an
arc $\psi:\Spec K((u))[[v]]\rightarrow X$ with
$\psi(0)=\psi(\eta)=\varphi(\eta)$. Now the result follows from
the fact that $\zeta(0)$ is contained in the Zariski closure of
$\zeta(\eta)$ in $\mathcal{L}(X/S)$.

Now we prove (2). By Proposition \ref{prelim}(2) it suffices to
deal with the case where $X$ is integral. Let $\xi$ be the generic
point of $X$. Since $k$ has characteristic zero, $\xi$ is formally
smooth over $k$ \cite[19.6.1]{ega4.1}, so $\mathcal{L}(\xi/k)$ is
irreducible by Corollary \ref{forsmooth}. We will show that this
set is dense in $\mathcal{L}(X/k)$.

Let $\psi$ be any point of $\mathcal{L}(X/k)$ and put
$x=\psi(\eta)$. We have to prove that $\psi$ belongs to the
Zariski closure of $\mathcal{L}(\xi/k)$. By (1) the point $\psi$
is a specialization of an element of $\mathcal{L}(x/k)$, so we may
assume that $\psi\in \mathcal{L}(x/k)$.

Denote by $k(X)$ the field of rational functions on $X$. The local
ring $\mathcal{O}_{X,x}$ can be dominated by a valuation $v$ on
$k(X)$. Denote by $V$ the valuation ring of $v$, and by $y$ the
closed point of $\Spec V$. The morphism of schemes $V\rightarrow
X$ induces a morphism of arc schemes $\mathcal{L}(\Spec
V/k)\rightarrow \mathcal{L}(X/k)$. Its image contains both
$\mathcal{L}(\xi/k)$ and $\mathcal{L}(x/k)$, by Theorem
\ref{diff}(4a). Hence, we may assume that $X=\Spec V$. Then
$\mathcal{L}(X/k)$ is integral by Lemma \ref{valuation}. We denote
by $\varphi$ the generic point of $\mathcal{L}(X/k)$. If
$\varphi(0)$ were contained in a strict closed subset $Z$ of $X$,
then we would have $\pi_0^{-1}(Z)=\mathcal{L}(X/k)$. This is
impossible because of the existence of the section $\sigma_{X/k}$.
Hence, $\varphi$ belongs to $\mathcal{L}(\xi/k)$, and this set is
dense in $\mathcal{L}(X/k)$.
\end{proof}

Theorem \ref{Qkol} generalizes \cite[3.3]{NiSe1} to arbitrary
$k$-schemes, avoiding the use of resolution of singularities.

\begin{cor}
For any morphism of schemes $X\rightarrow S$ and any closed
subscheme $Z$ of $X$ with complement $U$, the scheme
$\mathcal{L}(U/S)$ is dense in $\mathcal{L}(X/S)\setminus
\mathcal{L}(Z/S)$.
\end{cor}
\begin{proof}
This follows immediately from Theorem \ref{Qkol}(1).
\end{proof}

\begin{remark}
Even when $X$ is an irreducible complex variety, the jet schemes
$\mathcal{L}_n(X/\C)$ can be reducible. See \cite[0.1]{Mu} for the
relation with the nature of the singularities of $X$.
\end{remark}

\subsection{Wedge schemes}\label{sec-top-wedge}
\begin{definition}\label{wedge}
For any morphism of schemes $X\rightarrow S$ and any integer
$e\geq 0$, the $e$-th wedge scheme $\mathcal{L}^{(e)}(X/S)$ is
defined inductively by $\mathcal{L}^{(0)}(X/S)=X$ and
$$\mathcal{L}^{(e)}(X/S)=\mathcal{L}(\mathcal{L}^{(e-1)}(X/S)/S)$$
for $e>0$.
\end{definition}
By definition, for any local $S$-algebra $A$, there exists a
natural bijection
$$\mathcal{L}^{(e)}(X/S)(A)=X(A[[t_1,\ldots,t_e]])$$
If $A$ is a field, these objects are called $e$-wedges on $X$.
They play an important role in the study of the Nash problem
\cite[5.1]{reguera-curve}.

More generally, for any tuple $i\in (\N\cup\{\infty\})^e$ one can
define a $i$-jet scheme $\mathcal{L}^{i}(X/S)$ in the obvious way;
Definition \ref{wedge} corresponds to the case
$i=(\infty,\ldots,\infty)$. We leave the definition of the $i$-jet
schemes and the various truncation morphisms to the reader.
\begin{prop}
If $k$ is a field of characteristic zero, then for any irreducible
$k$-scheme $X$ and any integer $e>0$, the scheme of $e$-wedges
$\mathcal{L}^{(e)}(X/k)$ is irreducible.
\end{prop}
\begin{proof}
This follows from Theorem \ref{Qkol} by induction on $e$.
\end{proof}

\subsection{Approximation by curves}\label{sec-curves}
In order to get a better hold on the topology of arc schemes of
varieties over a field of arbitrary characteristic, we show in
this section that arcs can be approximated by algebraic curves.
This result, which has an independent interest, will be applied in
Proposition \ref{conn}. To simplify notation, we will henceforth
write $\mathcal{L}_n(X)$ and $\mathcal{L}(X)$ instead of
$\mathcal{L}_n(X/S)$ and $\mathcal{L}(X/S)$ if the base scheme $S$
is clear from the context.

\begin{lemma}\label{zar-dense}
For any field $k$ and any $k$-scheme of finite type $X$, the set
$\mathcal{L}(X)_{alg}$ is Zariski-dense in $\mathcal{L}(X)$.
\end{lemma}
\begin{proof}
By Proposition \ref{basic}(1) we may assume that $k$ is
algebraically closed. The sets $\pi_n^{-1}(U)$, with $U$ an open
subscheme of $\mathcal{L}_n(X)$, form a basis for the
Zariski-topology on $\mathcal{L}(X)$. Assume that $\pi_n^{-1}(U)$
is non-empty. Since the scheme $(\pi^m_n)^{-1}(U)$ is non-empty
for each $m\geq 0$, it must contain a $k$-rational point because
it is of finite type over $k$. Applying Greenberg's Approximation
Theorem (Theorem \ref{grapp}) to the $k[[t]]$-scheme of finite
type $X\times_k k[[t]]$, we see that $\pi_n^{-1}(U)$ contains a
$k$-rational point.
\end{proof}
\begin{remark}
Any point of $\mathcal{L}(X)_{alg}$ is closed in $\mathcal{L}(X)$.
 If $k$ is uncountable, then $\mathcal{L}(X)_{alg}$ coincides with
the set of closed points on $\mathcal{L}(X)$, but this does not
hold in general \cite[2.9]{Ishii-toric}.
\end{remark}


\begin{definition}
We denote by $Cu(X)$ the set of irreducible curves on $X$, and by
$\mathcal{L}cu(X)$ the subset
$$\tau_{X/k}(X_{alg})\cup \left(\bigcup_{C\in Cu(X)}\mathcal{L}(C)_{alg}\right)$$ of $\mathcal{L}(X)$.
\end{definition}

In other words, a point $x$ of $\mathcal{L}(X)$ belongs to
$\mathcal{L}cu(X)$ iff its residue field $k'$ is algebraic over
$k$ and the corresponding arc $\psi_x:\Spec k'[[t]]\rightarrow X$
maps the generic point of $\Spec k'[[t]]$ to a point of $X$ whose
residue field has transcendence degree at most one over $k$. Note
that
$$\tau_{X/k}(X_{alg})\subset \left(\bigcup_{C\in
Cu(X)}\mathcal{L}(C)_{alg}\right)$$ as soon as $X$ has no
zero-dimensional connected components, since any closed point on a
connected $k$-scheme of finite type $X$ of dimension $\geq 1$ is
contained in a curve on $X$.
\begin{theorem}\label{approx}
Let $k$ be any field, and $X$ a $k$-scheme of finite type. The arc
space $\mathcal{L}(X)$ is densely covered by curves in the
following sense: the set $\mathcal{L}cu(X)$ is $t$-adically dense
in $\mathcal{L}(X)_{alg}$. In particular, $\mathcal{L}cu(X)$ is
Zariski dense in $\mathcal{L}(X)$.
\end{theorem}
\begin{proof}
Let $k^{alg}$ be an algebraic closure of $k$, and put
$X'=X\times_k k^{alg}$. It is clear from the definition that the
natural surjective morphism $\mathcal{L}(X')\rightarrow
\mathcal{L}(X)$ maps $\mathcal{L}cu(X')$ onto $\mathcal{L}cu(X)$.
Hence, we may assume that $k$ is algebraically closed.

 Denote by
$\mathcal{O}$ the local ring of $\A^1_k=\Spec k[t]$ at the origin,
and by $\mathcal{O}^h$ its henselization. This is a henselian
discrete valuation ring, and its completion is isomorphic to
$k[[t]]$. Moreover, $\mathcal{O}$ is excellent since it is
essentially of finite type over a field \cite[7.8.3]{ega4.2}, so
$\mathcal{O}^h$ is excellent by \cite[18.7.6]{ega4.4}. Hence, we
can apply Greenberg's Approximation Theorem (Theorem \ref{grapp})
to the $\mathcal{O}^h$-scheme $X\times_k \mathcal{O}^h$.

It implies that for any point $x$ in $\mathcal{L}(X)(k)=X(k[[t]])$
and any integer $n\geq 0$, the image of $x$ under the natural map
$X(k[[t]])\rightarrow X(k[t]/t^{n+1})$ lifts to a point $x'$ in
$X(\mathcal{O}^h)$. Since $X$ is of finite type over $k$ and
$\mathcal{O}^h$ is a direct limit of essentially \'etale
extensions  of $\mathcal{O}$ \cite[18.6.5]{ega4.4}, the point $x'$
is defined over such an extension $\mathcal{O}'$. The generic
point of $\mathcal{O}'$ is finite over the generic point of
$\mathcal{O}$, so it has transcendence degree one over $k$, and
the scheme-theoretic closure of $x'$ in $X$ has dimension one.
This shows that $\mathcal{L}cu(X)$ is $t$-adically dense in
$\mathcal{L}(X)(k)$. By Lemma \ref{zar-dense}, $\mathcal{L}cu(X)$
is Zariski-dense in $\mathcal{L}(X)$.
\end{proof}

\subsection{Decomposition of arc spaces}\label{sec-decomp}
In this section, we consider the decomposition into irreducible,
resp. connected components, of arc schemes of schemes of finite
type over an arbitrary base field $k$.

\begin{lemma}\label{curve}
If $k$ is any field and $C$ is an irreducible (resp. connected)
$k$-curve, then $\mathcal{L}(C)$ is irreducible (resp. connected).
\end{lemma}
\begin{proof}
By Propositions \ref{basic}(5) and \ref{prelim}(3) we may assume
that $C$ is reduced and that $k$ is perfect. We denote by
$\widetilde{C}\rightarrow C$ the normalization map.

First, assume that $C$ is irreducible.  Denote by
$nN(C)=\{x_1,\ldots,x_n\}$ the finite set of points where $C$ is
not normal.
 The natural morphism $h:\mathcal{L}(\widetilde{C})\rightarrow
\mathcal{L}(C)$ is surjective: if $x$ is contained in
$\mathcal{L}(C)\setminus \mathcal{L}(nN(C))$ then the
corresponding arc lifts (uniquely) to $\widetilde{C}$ by the
valuative criterion for properness \cite[7.3.8]{ega2}. If $x$ is
contained in $\mathcal{L}(x_i/k)$ then $x$ is the constant arc
$\tau_{X/k}(x_i)$ at $x_i$ by Corollary \ref{cor-ft}, so $x$ lifts
to the constant arc at any point of the inverse image of $x_i$ in
$\widetilde{C}$.

Since $k$ is perfect and $C$ irreducible, the normalization
$\widetilde{C}$ is smooth and connected, so
$\mathcal{L}(\widetilde{C})$ is irreducible by Proposition
\ref{prelim}(1) and the same holds for $\mathcal{L}(C)$ by
surjectivity of $h$.

Now assume that $C$ is connected, and denote by
$C_1,\ldots,C_\ell$ its irreducible components. By the first part
of the proof, $\mathcal{L}(C_i)$ is irreducible, and hence
connected, for each $i$. By Proposition \ref{prelim}(2),
$\mathcal{L}(C)=\cup_{i}\mathcal{L}(C_i)$. But if $C_i\cap C_j\neq
\emptyset$, then $\mathcal{L}(C_i)\cap \mathcal{L}(C_j)\neq
\emptyset$ (consider the constant arc at any point of $C_i\cap
C_j$). Since $C$ is connected, this implies that $\mathcal{L}(C)$
is connected as well.
\end{proof}

\begin{corollary}
Let $k$ be any field and let $C$ a $k$-curve. The map $$Z\mapsto
\mathcal{L}(Z/k)_{red}$$ defines a bijection between the set of
connected (resp. irreducible) components of $C$ and the set of
connected (resp. irreducible) components of $\mathcal{L}(C)$ (all
taken with their reduced induced structure).
\end{corollary}

\begin{proof}
This is a direct consequence of Lemma \ref{curve} and Proposition
\ref{prelim}(2).
\end{proof}

%

\begin{theorem}
\label{gen case} Let $k$ be any field and $X$ a $k$-scheme of
finite type.
\begin{enumerate}
\item If $X$ is irreducible, then $\mathcal{L}(X)\setminus
  \mathcal{L}(nSm(X))$ is irreducible.
\item The map $$Z\mapsto \mathcal{L}(Z/k)_{red}$$ defines a
bijection between the set of irreducible components of $X$ which
are geometrically reduced, and the set of irreducible components
of $\mathcal{L}(X)\setminus \mathcal{L}(nSm(X))$.
\end{enumerate}
\end{theorem}

\begin{proof}
(1) It follows from Proposition \ref{prelim}(1) that
$\mathcal{L}(Sm(X))$ is irreducible, so the result follows from
Theorem \ref{Qkol}(1).

(2) Let $\{X_i\}_{i\in I}$ be the set of irreducible components of
$X$. Then $$\mathcal{L}(X)\setminus \mathcal{L}(nSm(X))=\cup_{i\in
I} \left(\mathcal{L}(X_i)\setminus \mathcal{L}(nSm(X))\right)$$ by
Proposition \ref{prelim}(2). Moreover, $\mathcal{L}(X_i)\setminus
\mathcal{L}(nSm(X))$ is open in $\mathcal{L}(X_i)\setminus
\mathcal{L}(nSm(X_i))$, and hence, by (1), an irreducible closed
subset of $\mathcal{L}(X)\setminus \mathcal{L}(nSm(X))$ for each
$i\in I$. Finally, for $i\neq j$ in $I$,
$$\mathcal{L}(X_i)\setminus \mathcal{L}(nSm(X))\subset
\mathcal{L}(X_j)\setminus \mathcal{L}(nSm(X))$$ implies that
$X_i\setminus nSm(X)\subset X_j$ by the existence of the section
$\tau_{X/k}$. This happens iff $X_i\subset nSm(X)$, \textit{i.e.}
iff $X$ is not geometrically reduced.
\end{proof}

\begin{cor}
\label{finite} If $k$ is any field and $X$ is a $k$-scheme of
finite type, then $\mathcal{L}(X)$ has finitely many irreducible
components.
\end{cor}
\begin{proof}
By Proposition \ref{prelim}(2,3), we may assume that $X$ is
integral and that $k$ is perfect. By induction on the dimension of
$X$, it suffices to show that
$\mathcal{L}(X)\setminus\mathcal{L}(Sing(X))$ has finitely many
irreducible components. But since $k$ is perfect, $Sing(X)=nSm(X)$
so $\mathcal{L}(X)\setminus\mathcal{L}(Sing(X))$ is irreducible by
Theorem \ref{gen case}.
\end{proof}

\begin{lemma}
\label{concurve} If $k$ is any field and $X$ is a connected
$k$-scheme of finite type of dimension $\geq 1$, then for any pair
of closed points $\{x,y\}$ on $X$, there exists a connected curve
$C$ on $X$ containing $x$ and $y$.
\end{lemma}

\begin{proof}
The statement is classical and left to the reader. See, for
example, \cite{Liu}, Exercise 8.1.5.
\end{proof}

\begin{prop}
\label{conn} Let $k$ be any field, and $X$ a $k$-scheme of finite
type.
\begin{enumerate}
\item If $X$ is connected, then $\mathcal{L}(X)$ is connected.
\item If $X$ is irreducible and $nSm(X)$ is a finite set of
points, then $\mathcal{L}(X)$ is irreducible. In particular, if
$k$ is perfect and $X$ is irreducible with only isolated
singularities, then $\mathcal{L}(X)$ is irreducible.
\end{enumerate}
\end{prop}

\begin{proof}
(1) Any connected component is a union of irreducible components.
By Corollary \ref{finite}, $\mathcal{L}(X)$ has finitely many
irreducible components. Therefore, $\mathcal{L}(X)$ has finitely
many connected components, and they are open.

 Assume that $U_1$ and $U_2$ are two
distinct connected components of $\mathcal{L}(X)$. By Theorem
\ref{approx} there exist curves $C_1$ and $C_2$ on $X$ such that
$\mathcal{L}(C_i)$ intersects $U_i$ for $i=1,2$. By Lemma
\ref{curve}, $\mathcal{L}(C_i)$ is contained in $U_i$. Let $x_i$
be a closed point on $C_i$ for $i=1,2$. By Proposition
\ref{concurve} there exists a connected curve $D$ on $X$
containing both $x_1$ and $x_2$. This means that $\mathcal{L}(D)$
intersects $U_1$ and $U_2$, which contradicts the fact that
$\mathcal{L}(D)$ is connected by Lemma \ref{curve}.

(2) By (1), $\mathcal{L}(X)$ is connected, and by Theorem \ref{gen
case}, $\mathcal{L}(X)\setminus\mathcal{L}(nSm(X))$ is
irreducible. Since $\mathcal{L}(nSm(X))$ is topologically a
discrete finite set of points by Corollary \ref{cor-ft}, this
implies that $\mathcal{L}(X)$ is irreducible.
\end{proof}

\subsubsection*{A counterexample over an imperfect field}
\label{sec-count} If $k$ is a perfect field and $X$ a $k$-scheme
of finite type, then $Sing(X)$ and $nSm(X)$ coincide. In this
case, Theorem \ref{gen case} was proven in \cite[2.9]{Reguera}
(the condition that $X$ is reduced seems to be missing in
\cite[2.9]{Reguera}).
 We'll now show that, over an arbitrary field, it is not possible to replace
 $nSm(X)$ by $Sing(X)$ in the statement of Theorem \ref{gen
   case}.

\bigskip

Note that if $k$ has positive characteristic, it is not difficult
to construct examples of irreducible $k$-varieties $X$ such that
$\mathcal{L}(X)$ is not irreducible (see
 \cite[Rmq.\,1]{NiSe1}, or \cite[Exercise IV.6/3.d)]{kolchin}). The following counterexample
 has the additional property that the variety $X$ is regular, and
 hence shows that \cite[2.9]{Reguera} does not extend to
 imperfect base fields.

\begin{theorem}
\label{count} If $k$ is an imperfect field, then there exists a
regular irreducible $k$-variety $X$ such that $\mathcal{L}(X)$ is
not irreducible.
\end{theorem}


\begin{proof}
Denote by $p$ the characteristic of $k$ and choose an element $a$
in $k-k^p$. Consider the polynomial
$$f=x^p+yz^p-a\ \in k[x,y,z]$$
It is clear that $f$ is irreducible. We denote by $X$ the
hypersurface in $\A^3_{k}=\Spec k[x,y,z]$ defined by $f$.

\textit{Claim $1$: $X$ is regular.} We only have to show that each
point of $nSm(X)$ is a regular point of $X$. The closed subscheme
$nSm(X)$ of $X$ is defined by a single equation $z=0$ and it is
regular of dimension one. Hence, any point of $nSm(X)$ is a
regular point of $X$.

\bigskip

\textit{Claim $2$: $\mathcal{L}(X)$ is reducible.} Consider the
polynomial algebra
$$A=k[x_i,y_i,z_i]_{i\geq 0}$$ over $k$.
 Developing the
expression $$f(\sum_{i\geq 0}x_it^i,\sum_{i\geq
0}y_it^i,\sum_{i\geq 0}z_it^i)$$ into a power series over $A$, we
obtain an expression
$$
\sum_{i\geq 0}F_i(x_0,y_0,z_0,\ldots,x_i,y_i,z_i)t^i.
$$
The arc scheme $\mathcal{L}(X)$ is the closed subscheme of the
infinite-dimensional affine space $\Spec A$ defined by the ideal
$I=(F_0,F_1,\ldots)$. Looking at the equation $F_1=0$ one sees
that the closed subset $V(z_0)=\pi^{-1}_0(nSm(X))$ of
$\mathcal{L}(X)$ contains the open subset $D(y_1)$ of
$\mathcal{L}(X)$. The set $D(y_1)$ is non-empty:~it contains the
maximal ideal
$$(x_0^p-a,y_0, z_0, x_1,y_1-1, z_1,x_i,y_i,z_i)_{i\geq
2}$$ of $A/I$.
 Since $V(z_0)\neq
\mathcal{L}(X)$ (it is disjoint from $\tau_{X/k}(Sm(X))$), we can
conclude that $\mathcal{L}(X)$ is reducible.
\end{proof}
\begin{remark}
Note that a variety $X$ as in the statement of Theorem \ref{count}
must have dimension $\geq 2$, by Corollary \ref{cor-ft} and Lemma
\ref{curve}.
\end{remark}
\bibliographystyle{hplain}
\bibliography{wanbib,wanbib2}
\end{document}